\documentclass[10pt,reqno]{amsart}
\usepackage[letterpaper,margin=1in,footskip=0.25in]{geometry}
\usepackage[only,llbracket,rrbracket]{stmaryrd}
\usepackage{amssymb,mathscinet}
\usepackage{microtype}
\usepackage{mathtools}
\usepackage{etoolbox}
\usepackage{enumitem}
\PassOptionsToPackage{dvipsnames}{xcolor}
\usepackage{trimclip}
\usepackage{tikz-cd}
\usepackage[noadjust]{cite}
\renewcommand{\citepunct}{;\penalty\citemidpenalty\ }

\PassOptionsToPackage{pdfusetitle,colorlinks,pagebackref}{hyperref}
\usepackage{bookmark}
\hypersetup{citecolor=OliveGreen,linkcolor=Mahogany,urlcolor=Plum}
\usepackage[hyphenbreaks]{breakurl}

\newtheorem{theorem}{Theorem}[section]
\newtheorem{proposition}[theorem]{Proposition}
\newtheorem{question}[theorem]{Question}
\newtheorem{corollary}[theorem]{Corollary}
\newtheorem{lemma}[theorem]{Lemma}
\newtheorem{innercustomthm}{Theorem}
\newenvironment{customthm}[1]
  {\renewcommand\theinnercustomthm{#1}\innercustomthm}
  {\endinnercustomthm}
\newtheorem{innercustomcor}{Corollary}
\newenvironment{customcor}[1]
  {\renewcommand\theinnercustomcor{#1}\innercustomcor}
  {\endinnercustomcor}

\theoremstyle{definition}
\newtheorem{definition}[theorem]{Definition}
\newtheorem{convention}[theorem]{Convention}
\newtheorem{remark}[theorem]{Remark}
\newtheorem{example}[theorem]{Example}

\DeclareMathOperator{\Frac}{Frac}
\DeclareMathOperator{\im}{im}
\newcommand{\QQ}{\mathbf{Q}}
\newcommand{\RR}{\mathbf{R}}
\newcommand{\ZZ}{\mathbf{Z}}
\newcommand{\fm}{\mathfrak{m}}
\newcommand{\id}{\mathrm{id}}

\DeclarePairedDelimiterX\abs[1]\lvert\rvert{\ifblank{#1}{\:\cdot\:}{#1}}
\DeclarePairedDelimiterX\norm[1]\lVert\rVert{\ifblank{#1}{\:\cdot\:}{#1}}
\providecommand\given{}
\newcommand\SetSymbol[1][]{\nonscript : \allowbreak \nonscript \mathopen{}}
\DeclarePairedDelimiterX\Set[1]\{\}{\renewcommand\given{\SetSymbol[\delimsize]} #1}

\newcommand{\hooklongrightarrow}{\lhook\joinrel\longrightarrow}
\makeatletter
\renewcommand{\twoheadrightarrow}{\mathrel{\text{\two@rightarrow}}}
\newcommand{\longtwoheadrightarrow}{\mathrel{\text{\longtwo@rightarrow}}}
\newcommand{\two@rightarrow}{\sbox0{$\m@th\rightarrow$}\smash{\rlap{\kern0.1\wd0 \clipbox{{.3\width} {-\height} 0pt {-\height}}{$\m@th\rightarrow$}}}$\m@th\rightarrow$}
\newcommand{\longtwo@rightarrow}{\sbox0{$\m@th\longrightarrow$}\smash{\rlap{\kern0.175\wd0 \clipbox{{.3\width} {-\height} 0pt {-\height}}{$\m@th\longrightarrow$}}}$\m@th\longrightarrow$}
\makeatother

\begin{document}
\title[Tate algebras and Frobenius non-splitting of excellent regular rings]{Tate algebras and \\  Frobenius non-splitting of excellent regular rings}

\author{Rankeya Datta}
\thanks{The first author was supported by an AMS-Simons travel grant.}
\address{Department of Mathematics\\Michigan State University\\East Lansing, IL 48824\\USA}
\email{\href{mailto:dattaran@msu.edu}{dattaran@msu.edu}}
\urladdr{\url{https://www.rankeyadatta.com}}

\author{Takumi Murayama}
\address{Department of Mathematics\\Princeton University\\Princeton, NJ
08544-1000\\USA}
\email{\href{mailto:takumim@math.princeton.edu}{takumim@math.princeton.edu}}
\urladdr{\url{https://web.math.princeton.edu/~takumim/}}
\thanks{The second author was supported by the National Science
Foundation under Grant No.\ DMS-1902616}

\makeatletter
  \hypersetup{
    pdfauthor={Rankeya Datta and Takumi Murayama},
    pdfsubject=\@subjclass,
    pdfkeywords={Tate algebra, Frobenius splitting, F-purity, p\textasciicircum-1-linear map, excellent ring, convergent power series}
  }
\makeatother

\begin{abstract}
An excellent ring of prime characteristic for which the
Frobenius map is pure is also Frobenius split in many commonly 
occurring situations in positive characteristic commutative
algebra and algebraic geometry.
However, using a fundamental construction from rigid geometry, we 
show that excellent $F$-pure rings of prime characteristic are not Frobenius split in general,
even for Euclidean domains.
Our construction uses the existence of a complete non-Archimedean field
$k$ of characteristic $p$ with no nonzero continuous $k$-linear maps $k^{1/p}
\to k$. An explicit example of such a field is given based on ideas of Gabber,
and may be of independent interest.
Our examples settle
a long-standing open question in the theory of $F$-singularities whose origin can be 
traced 
back to when Hochster and Roberts introduced the notion of $F$-purity. The excellent 
Euclidean domains we construct also admit
no nonzero $R$-linear maps $R^{1/p} \rightarrow R$. 
These are the first examples that illustrate that $F$-purity
and Frobenius splitting 
define different classes of singularities for excellent domains, and
are also the first examples of 
excellent domains 
with no nonzero $p^{-1}$-linear maps. The latter is particularly
interesting from the perspective of the theory of test ideals.
\end{abstract}

\maketitle

\section{Introduction}\label{sect:intro}
Let $R$ be a Noetherian ring of prime characteristic $p > 0$.
A key theme in characteristic $p$ commutative algebra and algebraic
geometry is to use the \textsl{Frobenius map}
\[
  \begin{tikzcd}[column sep=1.475em,row sep=0]
    \mathllap{F_R\colon} R \rar & F_{R*}R\\
    r \rar[mapsto] & r^p
  \end{tikzcd}
\]
to study the singularities of $R$ (see the surveys \cite{SZ15,PST17,TW18}).
This line of study was initiated by Kunz, who proved that $R$ as above is
regular if and only if $F_R$ is faithfully flat \cite[Thm.\ 2.1]{Kun69}.
Following Kunz's results and building on their work showing that 
rings of invariants of reductive groups are
Cohen--Macaulay \cite{HR74},
Hochster and Roberts defined $R$ to be \textsl{$F$-pure} if $F_R$ is
pure in the following sense: for every $R$-module $M$, the base
change $M \rightarrow M \otimes_R F_{R*}R$
is injective \cite[p.\ 121]{HR76}.
Since faithfully flat maps are pure \cite[Thm.\ 7.5$(i)$]{Mat89},
it follows that regular rings of characteristic $p$ are $F$-pure. In particular,
$F$-pure rings form a natural class of singular characteristic $p$ 
Noetherian rings.

\par In their study of the cohomology of Schubert varieties, Mehta and
Ramanathan defined $R$ to be \textsl{Frobenius split} if $F_R$ splits
as a map of $R$-modules, that is, if $F_R$ admits
an $R$-linear left inverse \cite[Def.\ 2]{MR85}.
Since split maps are automatically pure, Frobenius split rings are
$F$-pure.
Whether the converse holds in nice situations was a mystery,
prompting the following folklore question (see for instance \cite[Rem.\ 1.12]{SZ15}):

\begin{question}\label{ques:fpureisfsplit}
  Let $R$ be an \emph{excellent} Noetherian ring of prime characteristic.
  If $R$ is $F$-pure, is it necessarily Frobenius split?
\end{question}

\noindent For the definition of an excellent ring, we refer 
the reader to \cite[D\'ef.\ 7.8.2]{EGAIV2}. 

\medskip
\par Our main result shows that even excellent \emph{regular} rings are not
Frobenius split in general.
\begin{customthm}{\ref{thm:mainthm}}
For every prime $p > 0$,
there exists a complete non-Archimedean field $(k,\abs{})$ of
characteristic $p$ such that the Tate algebra 
$T_n(k) \coloneqq k\{ X_1,X_2,\ldots,X_n \}$
is not Frobenius split for each $n > 0$.
In fact, $T_n(k)$ admits no nonzero $T_n(k)$-linear maps
$F_{T_n(k)*}T_n(k) \to T_n(k)$ for each $n > 0$.
\end{customthm} 
The origin of Question \ref{ques:fpureisfsplit} can be traced to the 
aforementioned work of 
Hochster and Roberts on the purity of Frobenius, where the equivalence of
$F$-purity and Frobenius splitting is observed when 
$F_R$
is a finite map \cite[Cor.\ 5.3]{HR76}. Finiteness of Frobenius
implies $R$ is excellent by \cite[Thm.\ 2.5]{Kun76}. Question
\ref{ques:fpureisfsplit} also has an affirmative answer when $R$ is
essentially of finite type over a complete local ring by \cite[Thm.\ 3.1.1]{DM}.
Thus, the notions of $F$-purity and Frobenius splitting coincide for all rings
appearing in the study of algebraic varieties over positive characteristic fields.

\par The Tate algebra $T_n(k)$, introduced by Tate
\cite{Tat71}, is the analogue of the polynomial ring 
$k[X_1,\ldots, X_n]$ in the world of rigid analytic geometry and shares many 
of its good properties (see \cite{Bos14} and Theorem \ref{thm:tateprops}). For example,
$T_n(k)$ is an $n$-dimensional Jacobson unique factorization domain 
that is regular (hence $F$-pure), and is excellent by a theorem of Kiehl
\cite[Thm.\ 3.3]{Kie69}.
Theorem \ref{thm:mainthm} says that despite the similarities between $T_n(k)$
and $k[X_1,\dots,X_n]$, rigid analytic 
geometry and classical algebraic geometry are fundamentally different 
from the perspective of Frobenius splittings.
Theorem \ref{thm:mainthm} should also be contrasted with 
\cite[Thm.\ 3.2]{DS18}, where Smith and the first author show that if $R$ is a
Noetherian domain of characteristic $p$ whose fraction field $K$ satisfies 
$[K^{1/p}:K] < \infty$, then $R$
is excellent precisely when $R$ admits a nonzero $R$-linear
map $F_{R*}R \rightarrow R$.

\medskip
\par We establish Theorem \ref{thm:mainthm}
by first proving the following
necessary and sufficient criterion for Frobenius splitting of the Tate algebra $T_n(k)$ in terms
of a topological property of the non-Archimedean field $k$.

\begin{customthm}{\hyperref[thm:refined-mainthm]{B}}
  \label{thm:refined-mainthmintro}
  Let $(k,\abs{})$ be a complete non-Archimedean field of
  characteristic $p > 0$.
  The following are equivalent:
  \begin{itemize}
    \item
      $T_n(k)$ is Frobenius split for each integer $n > 0$.
    \item 
      $T_n(k)$ 
      has a nonzero $T_n(k)$-linear map
      $F_{T_n(k)*} T_n(k) \rightarrow T_n(k)$ for each integer $n > 0$.
    \item
      There exists a nonzero continuous $k$-linear map $f\colon
      k^{1/p} \to k$.
  \end{itemize}
\end{customthm}

\noindent The equivalent statements above are a subset of those we prove in Theorem
\ref{thm:refined-mainthm}.
Here, we view $k^{1/p}$ as a normed $k$-vector space with the canonical 
norm that extends the norm on $k$, and then equip $k^{1/p}$ and $k$ with
the corresponding metric topologies to be able to talk about continuous maps.

\par Using this topological characterization,
Theorem \ref{thm:mainthm} follows by explicitly constructing 
a non-Archimedean field $(k,\abs{})$ for which
$k^{1/p}$ admits no nonzero continuous linear functionals (see
Theorem \ref{thm:main-example}). The existence of such fields
is suggested by Gerritzen \cite{Ger67} and Kiehl \cite{Kie69} (see also
\cite[p.\ 63]{BGR84}),
who take significant care to show that the Tate algebra is Japanese and
excellent
when $[k^{1/p}:k]$ is infinite. However, we were unable to locate an explicit example
in the literature, and the example in Theorem \ref{thm:main-example} is due to Ofer Gabber.
Similar constructions of valuative fields with infinite $p$-degree
 have been studied by Blaszczok and Kuhlmann \cite{BK15},
although the connection with the
existence of continuous functionals was not made (see Remark \ref{rem:BKfields}).

\medskip
\par We would like to isolate one consequence of Theorem \ref{thm:refined-mainthm}
to emphasize the simplicity of the examples obtained in this paper.

\begin{customcor}{\ref{cor:no-p-1-linear-maps}}
There exists an excellent Euclidean domain $R$ of characteristic
$p > 0$ such that $R$ admits no nonzero $R$-linear maps $F_{R*}R \rightarrow R$.
Moreover, one can choose $R$ to be local and Henselian as well.
\end{customcor}

\noindent Corollary \ref{cor:no-p-1-linear-maps} follows from Theorem \ref{thm:refined-mainthm}
and the aforementioned example of a non-Archimedean field $k$ with no nonzero
continuous linear functionals $k^{1/p} \rightarrow k$ upon taking $R$ to be the
Tate algebra $T_1(k)$. For this, one uses the well-known fact that  
like $k[X]$, the Tate algebra $T_1(k)$ is a Euclidean
domain (see Theorem \ref{thm:tateprops}$(\ref{thm:T_1euclid})$). In addition, we 
obtain a local and Henselian example by proving an analogue of
Theorem \ref{thm:refined-mainthm} for \textsl{convergent power series rings}
(see Definition \ref{def:convergent-series}) in Theorem \ref{prop:local-analogue}.
Corollary \ref{cor:no-p-1-linear-maps} should be contrasted with the well-known
fact that a Noetherian complete local $F$-pure ring is always Frobenius split
(see Remark \ref{rem:as-close-to-complete}). 
Thus, our local Henselian example shows that Question \ref{ques:fpureisfsplit} fails even
for the class of excellent Noetherian local rings that are closest in behavior to complete
local rings.

\par Maps of the form
\begin{equation}\label{eq:p-1-linearmap}
  F_{R*}R \longrightarrow R
\end{equation}
are used extensively in prime characteristic
commutative algebra and algebraic geometry.
One of the first such examples comes from looking at the stalks of the
\textsl{Cartier operator,} which
is a map of the form $F_{X*}\omega_X \to \omega_X$ on a smooth variety $X$ over a
perfect field of prime characteristic \cite{Car57}.
Blickle and B\"ockle later called maps of the form \eqref{eq:p-1-linearmap}
\textsl{$p^{-1}$-linear maps} \cite[p.\ 86]{BB11}.
The reason for this terminology is that for
$r \in R$ and $x \in F_{R*}R$, a $p^{-1}$-linear map $\phi$ satisfies 
\[\phi(r^{p}x) = \phi(r \cdot x) = r\phi(x),\]
where the first equality follows from the $R$-module
structure on $F_{R*}R$ induced by restriction of scalars via the Frobenius
map $F_R$.

\par The notion of a $p^{-1}$-linear map was used by Hara and Takagi \cite{HT04},
Schwede \cite{Sch10}, and Blickle \cite{Bli13}, among others, to give an alternate
approach to the theory of \textsl{test ideals}.
Test ideals were originally defined by
Hochster and Huneke in their celebrated theory of tight closure \cite{HH90},
and later extended to the context of pairs by Hara and Yoshida \cite{HY03} and
Takagi \cite{Tak04}.
Quite surprisingly, test ideals were shown to be related to the characteristic 
zero notion of multiplier ideals
by work of Smith \cite{Smi00}, Hara \cite{Har01}, Hara and Yoshida \cite{HY03},
and Takagi \cite{Tak04}.
Since then, algebraists and geometers
have used test ideals to establish many prime characteristic analogues of
characteristic $0$ results whose proofs require deep vanishing theorems that
are known to fail in positive characteristic \cite{Ray78,HK15}.
See the surveys \cite{ST12,BS13,SZ15,PST17,TW18} for various applications.
More recently,
Schwede's insight \cite{Sch10} to use compatibly split subschemes 
defined via $p^{-1}$-linear maps as a foundation for the theory of test ideals,
building on the work of Mehta and Ramanathan \cite{MR85} and
Lyubeznik and Smith \cite{LS01},
has 
further advanced our understanding of positive characteristic rings and varieties.
In this context, Corollary \ref{cor:no-p-1-linear-maps} serves to 
caution us that although excellent
rings have nonzero $p^{-1}$-linear maps in many geometric situations, they may fail
to have such maps in general. In particular,
the approach to test ideals using $p^{-1}$-linear maps
will require a different formulation if one hopes to generalize the theory to
arbitrary excellent rings and schemes over $\mathbf{F}_p$. Even an extension
of the theory encompassing Tate algebras and their quotients
may need a new approach.

\medskip
\par While our emphasis so far has been on non-existence results, Theorem \ref{thm:refined-mainthm}
has positive consequences as well. For instance, we can show that $T_n(k)$ is Frobenius split
for many commonly occurring non-Archimedean fields of positive characteristic.

\begin{customcor}{\ref{cor:F-split-usually}}
  Let $(k,\abs{})$ be a complete non-Archimedean field of
  characteristic $p > 0$.
For each $n > 0$, the Tate algebra $T_n(k)$ is Frobenius split in the
following cases:
\begin{enumerate}[label=$(\roman*)$,ref=\roman*]
	\item $(k,\abs{})$ is spherically complete.
	\item $k^{1/p}$ has a dense $k$-subspace $V$ that has a countable $k$-basis, hence in particular
	if $[k^{1/p}:k] < \infty$.
	\item $\abs{k^\times}$ is not discrete, and the norm on $k^{1/p}$ is polar.
\end{enumerate}
\end{customcor}
\noindent 
A spherically complete non-Archimedean field 
(see Definition \ref{def:spherically-complete})
is a generalization of a field equipped with a discrete valuation whose 
corresponding valuation ring is complete.
The notion of polarity in $(\ref{cor:F-split-polar})$ is a more technical condition
due to Schikhof generalizing $(\ref{cor:F-split-usually-spher})$ and
$(\ref{cor:F-split-usually-count})$, which appears in Theorem
\ref{thm:Hahn-Banach}$(\ref{thm:Hahn-Banachpolar})$.
The proof of Corollary \ref{cor:F-split-usually} requires some knowledge of 
non-Archimedean functional analysis that we will summarize in Subsection
\ref{subsection:NA-functional-analysis}.
The essential fact is that when $k$ is a non-Archimedean
field of the above two types, then variants of the Hahn--Banach
theorem for normed spaces over $\mathbf{R}$ or $\mathbf{C}$ also hold for
normed spaces over $k$. In particular, this gives a wealth of nonzero
continuous linear functionals of $k^{1/p}$, thereby allowing us to use
Theorem \ref{thm:refined-mainthm}.

\subsection*{Organization of the paper} In Section \ref{sec:NA-review}, we
review all definitions and results from non-Archimedean analysis that
are used in the rest of the paper. Our aim in writing this section has been
to motivate some of the constructions of non-Archimedean geometry by relating
them to more familiar constructions for polynomial rings.
In Section \ref{sec:refined-mainthm}, we first prove Theorem \ref{thm:refined-mainthm}
and then use it to deduce Frobenius splitting of Tate algebras in some cases
in Corollary \ref{cor:F-split-usually}. In Section \ref{sec:local-construction}, 
we adapt the proof of
Theorem \ref{thm:refined-mainthm} to the local setting of convergent power series
rings (see Theorem \ref{prop:local-analogue}), 
that later gives us a wealth of local examples for which Question \ref{ques:fpureisfsplit}
fails.  Finally, Section 
\ref{sec:troublesome-field} contains Gabber's example of a non-Archimedean field
$k$ such that $k^{1/p}$ admits no nonzero continuous linear functionals. Theorem \ref{thm:mainthm}
and Corollary \ref{cor:no-p-1-linear-maps} are then proved as straightforward
consequences of Theorem \ref{thm:refined-mainthm} by working over the
aforementioned example.
 Their analogues for the local ring of convergent power series are observed as well.

\subsection*{Notation}
All rings will be commutative with identity.
If $R$ is a ring of prime characteristic $p > 0$, then the \textsl{Frobenius
map} on $R$ is the ring map 
\[
  \begin{tikzcd}[column sep=1.475em,row sep=0]
    \mathllap{F_R\colon} R \rar & F_{R*}R.\\
    r \rar[mapsto] & r^p
  \end{tikzcd}
\]
The notation $F_{R*}R$ is used to emphasize the fact that the target of the Frobenius
map has the (left) $R$-algebra structure given by $r \cdot x = r^px$.
A \textsl{$p^{-1}$-linear map on $R$} is the datum of an $R$-linear map $F_{R*}R
\rightarrow R$.
Thus, a Frobenius splitting of $R$ is a $p^{-1}$-linear map on $R$ that sends
$1$ to $1$.
\par If $R$ is a domain, we will sometimes identify $F_{R*}R$ with the ring
$R^{1/p}$ of $p$-th roots of elements in $R$.
Under this identification, the Frobenius map $F_R\colon R \to F_{R*}R$
corresponds to the inclusion $R \hookrightarrow R^{1/p}$.

\subsection*{Acknowledgments}
We are first and foremost extremely grateful to Ofer Gabber for allowing
us to reproduce his example of a non-Archimedean field $k$ such that $k^{1/p}$ admits no nonzero
continuous $k$-linear functionals. 
We are also grateful to him for his insightful comments on drafts of this paper.
We would next like to thank Franz-Viktor Kuhlmann for bringing our attention to
the examples of valuative fields with infinite $p$-degree in \cite{BK15}.
The first author learned about Question \ref{ques:fpureisfsplit} from
 Karl Schwede at the 2015 Mathematics Research Communities 
in commutative algebra and is grateful to Karl for numerous discussions
on this problem back then.
The first author would also like to thank Karen E. Smith for fruitful
discussions about the existence of nonzero $p^{-1}$-linear maps that began
while writing \cite{DS18}.
Eric Canton and Matthew Stevenson thought about some of these questions with
us and we thank them for their insights as well.
We are grateful to Benjamin Antieau, Oren Ben-Bassat, Bhargav Bhatt,
Brian Conrad, Remy van Dobben de Bruyn, Mattias Jonsson, Kiran S. Kedlaya,
and Salma Kuhlmann for helpful
conversations, and to Melvin Hochster, Linquan Ma, Mircea Musta\c{t}\u{a},
Karl Schwede, Karen E. Smith,
and Kevin Tucker for their comments on previous drafts of this paper
and for illuminating conversations pertaining to the origin of
Question \ref{ques:fpureisfsplit}. Finally, we thank the referees for their
helpful comments.

\section{A review of non-Archimedean analysis}\label{sec:NA-review}
We will use basic results from non-Archimedean functional analysis to
produce our counterexamples and to prove that the Tate algebra is
Frobenius split in some commonly occurring situations. Thus, we collect all relevant 
definitions and results we will need 
in the present section
to make the paper easier to navigate for readers unfamiliar with the
non-Archimedean world.
All results appearing in this section are well-known.

\subsection{Non-Archimedean fields, normed spaces, and Tate algebras}
\begin{definition}
  A \textsl{real-valued field} $(k,\abs{})$ is a pair consisting of a
  field $k$ and a non-Archimedean valuation $\abs{}\colon k \to \RR_{\ge 0}$
  (written multiplicatively),
  which is a function that satisfies:
  \begin{itemize}
    \item $\abs{x} = 0$ if and only if $x = 0$;
    \item $\abs{x+y} \le \max\{\abs{x},\abs{y}\}$; and
    \item $\abs{xy} = \abs{x} \cdot \abs{y}$.
  \end{itemize}
  
  The \textsl{valuation ring of $k$} is the subring
  $k^\circ \coloneqq \Set{x \in k \given \abs{x} \leq 1}$. 
  This is a local ring with maximal ideal $k^{\circ\circ}
  \coloneqq \Set{x \in k \given \abs{x} < 1}$. The \textsl{value group of 
  $k$} is the group $\abs{k^\times}$. 
\end{definition}
  
\begin{convention} 
  \label{conv:nontrivial}
  By a \textsl{non-Achimedean field} we mean a real-valued field that is
  complete with respect to the metric $\abs{x-y}$ induced by $\abs{}$, and
  whose value group $\abs{k^\times}$ is not the trivial group.
  This latter assumption implies that $k^{\circ}$ has Krull dimension $1$ \cite[Thm.\ 10.7]{Mat89}.
  \end{convention}

Throughout this paper, we will only consider real-valued fields that are non-Archimedean, with the exception of the auxiliary field $M$ in Theorem 
\ref {thm:main-example}.

\medskip
  
\par We will use the following well-known property of valuations.
  
\begin{lemma}[see {\cite[Prop.\ 2.1/2]{Bos14}}]
  \label{lem:max-norm}
  Let $(k,\abs{})$ be a non-Archimedean field. For $x, y \in k$, if 
  $\abs{x} \neq \abs{y}$, then
  \[
    \abs{x+y} = \max\bigl\{\abs{x},\abs{y}\bigr\}.
  \]
\end{lemma}
  
  \par We next define the analogue of a vector space over a non-Archimedean field.
  
\begin{definition}
  \label{def:normed-space}
  Let $(k,\abs{})$ be a non-Archimedean field. A \textsl{normed space} $(E,\norm{})$
  over $k$ is a $k$-vector space $E$ with a \textsl{norm} $\norm{}\colon E \rightarrow \RR_{\ge 0}$
  that satisfies the following properties:
  \begin{itemize}
  \item $\norm{x} = 0$ if and only if $x = 0$;
  \item $\norm{x+y} \leq \max\{\norm{x},\norm{y}\}$; and
  \item If $c \in k$ and $x \in E$, then $\norm{cx} = \abs{c} \cdot \norm{x}$.
  \end{itemize}
  If $E$ is complete in the metric induced by $\norm{}$, then $E$ is called a
  \textsl{Banach space} over $k$.
  
  \par A \textsl{Banach $k$-algebra} $(A,\norm{})$ is a $k$-algebra $A$ such that
  $(A,\norm{})$ is a $k$-Banach space and such that the norm $\norm{}$ satisfies the following
  additional property:
  \begin{itemize}
  \item $\norm{xy} \leq \norm{x} \cdot \norm{y}$.
  \end{itemize}
  The norm $\norm{}$ is \textsl{multiplicative} if equality holds in the above
  inequality.
\end{definition}
  
  \par All Banach algebras considered in this paper will be multiplicative.
  
  \begin{remark}
  \label{rem:extending-norms}
  \leavevmode
  \begin{enumerate}[label=$(\roman*)$,ref=\roman*]
  \item\label{fin-dim-norm} A finite-dimensional vector space $E$ over a 
  non-Archimedean field $(k,\abs{})$
  can always be given the structure of a $k$-Banach space in a canonical way. 
  Indeed, if $x_1,x_2,\dots,x_n$
  is a basis of $E$, then for every element $x = \sum_i a_ix_i$ where $a_i \in k$, 
  one can define
  \[
    \norm{x} \coloneqq \max_{1 \le i \le n}\bigl\{\abs{a_i}\bigr\}.
  \]
  Now, because $k$ is
  complete with respect to $\abs{}$, 
  every norm on $E$ is equivalent to the one just defined, even
  though this norm depends on the choice of the basis \cite[App.\ A, Thm.\ 1]{Bos14}.
  It is fairly straightforward to verify that $E$ is complete in the
  above norm.
  
  \item\label{alg-ext-norm} When $\ell$ is an algebraic extension of $k$, by expressing
  $\ell$ as a filtered colimit of finite subextensions, one can show that there
  exists a unique (not just equivalent) norm on $\ell$ that extends the norm
  on $k$ \cite[App.\ A, Thm.\ 3]{Bos14}. However, if $[\ell:k] = \infty$, 
  then $\ell$ need not be complete with respect to the metric induced by this norm.
  \end{enumerate}
  \end{remark}
  
   \par The principal example
  of a Banach $k$-algebra is the Tate algebra.
  
\begin{definition}[see {\cite[Def.\ 2.2/2]{Bos14}}]
\label{def:Tate-alg}
  Let $(k,\abs{})$ be a non-Archimedean field.
  For every positive integer $n > 0$, the \textsl{Tate algebra} in $n$
  indeterminates over $k$ is the $k$-subalgebra
  \[
    T_n(k) \coloneqq k\{ X_1,X_2,\ldots,X_n \}
    \coloneqq
    \Set[\Bigg]{\sum_{\nu \in \ZZ_{\ge0}^n} a_\nu X^{\nu} \given
      \begin{array}{@{}c@{}}
        a_\nu \in k\ \text{and}\ \abs{a_\nu} \to 0\\
        \text{as}\ \nu_1 + \nu_2 + \cdots + \nu_n \to \infty
      \end{array}}
  \]
  of
   the formal power
  series $k\llbracket X_1,X_2,\ldots,X_n \rrbracket$ in $n$ indeterminates over $k$.
   An element of $T_n(k)$ is called a \textsl{restricted power series}. For $n = 1$,
  we will denote the indeterminate in $T_1(k)$ by just $X$ instead of $X_1$.

	\par The Tate algebra becomes a $k$-Banach algebra when equipped with 
	the \textsl{Gauss norm}, which is defined as follows. 
	For every element $\sum_{\nu \in \ZZ_{\ge 0}^n} a_\nu
X^{\nu} \in T_n(k)$, we set
\[
  \norm[\Bigg]{\sum_{\nu \in \ZZ_{\ge 0}^n} a_\nu X^{\nu}} \coloneqq
  \max_{\nu \in \ZZ_{\ge0}^n}\bigl\{\abs{a_\nu}\bigr\}.
\]
One can show that the Gauss norm is multiplicative \cite[pp.\ 13--14]{Bos14}.
\end{definition}

\par The remarkable fact is that $T_n(k)$ shares many properties of the polynomial ring 
$k[X_1,\dots,X_n]$. We collect these properties for
readers who may be unfamiliar with $T_n(k)$.

\begin{theorem}\label{thm:tateprops}
  Let $(k,\abs{})$ be a non-Archimedean field, and let $n$ be a
  positive integer.
  Then, the Tate algebra $T_n(k)$ satisfies the following properties:
  \begin{enumerate}[label=$(\roman*)$,ref=\roman*]
    \item\label{thm:tatenoeth} $T_n(k)$ is Noetherian.
    \item\label{thm:tateufd} $T_n(k)$ is a unique factorization domain.
    \item\label{thm:tatejacobson} $T_n(k)$ is Jacobson, that is, every
      radical ideal is the intersection of the maximal ideals containing it.
    \item\label{thm:tatekrulldim}
      All maximal ideals of $T_n(k)$ are generated by $n$ elements and have
      height $n$. In particular, $T_n(k)$ has Krull dimension $n$.
  \item\label{thm:tatereg} $T_n(k)$ is regular.
  \item\label{thm:T_1euclid} $T_1(k)$ is a Euclidean domain with  
  associated Euclidean function $T_1 \smallsetminus \{0\} \rightarrow \mathbf{Z}_{\geq 0}$
  given by mapping a restricted power series $f = \sum_{i=0}^\infty a_iX^i$
  to the largest index $N$ such that $\abs{a_N} = \norm{f}$.
  \item\label{thm:kiehlexcellent} $T_n(k)$ is excellent.
    \end{enumerate}
\end{theorem}

\begin{proof}[Indication of proof]
  $(\ref{thm:tatenoeth})$--$(\ref{thm:tatekrulldim})$ are proved in
  \cite[Props.\ 2.2/14--17]{Bos14}.
  $(\ref{thm:tatereg})$ follows from $(\ref{thm:tatekrulldim})$, since the
  latter implies that all localizations of $T_n(k)$ at maximal ideals
  are regular local. $(\ref{thm:T_1euclid})$ follows from \cite[Cor.\ 2.2/10]{Bos14}.
  Finally, $(\ref{thm:kiehlexcellent})$ is the hardest property to show,
and can be found in \cite[Thm.\ 3.3]{Kie69} (see also \cite[\S1.1]{Con99}). Key ingredients in the proofs
of all these properties are the Weierstrass Division and the Weierstrass Preparation
theorems \cite[Thm.\ 2.2/8 and Cor.\ 2.2/9]{Bos14}.
\end{proof}

\par For proofs of the Weierstrass Division and Preparation theorems, one
works with certain distinguished elements in $T_n(k)$ that
we now introduce. These elements should be thought of as power series analogues
of essentially monic polynomials (with respect to one of the variables) in polynomial rings. 

\begin{definition}[{see \cite[Def.\ 2.2/6]{Bos14}}]
\label{def:Weierstrass-dist}
A restricted power series $g = \sum_{\nu=0}^\infty g_\nu X^\nu_n \in T_n(k)$
with coefficients $g_\nu \in T_{n-1}(k)$ is called \textsl{$X_n$-distinguished
of order $s \in \mathbf{Z}_{\geq 0}$} if the following hold:
\begin{itemize}
  \item $g_s$ is a unit in $T_{n-1}(k)$; and
	\item $\norm{g_s} = \norm{g}$ and $\norm{g_s} > \norm{g_\nu}$ for every $\nu > s$.
\end{itemize}
\end{definition}

\begin{remark}\label{rem:units}
\leavevmode
\begin{enumerate}[label=$(\roman*)$,ref=\roman*]
    \item \label{units} An element $f \in T_n(k)$ is a unit if and only if the constant term
    of $f$ has absolute value strictly bigger than the absolute value of the coefficients 
    of all other terms of $f$ \cite[Cor.\ 2.2/4]{Bos14}. In the language of distinguished
    elements, $f \in T_n(k)$ is a unit precisely when it is $X_n$-distinguished of order $0$.

	\item \label{terms} An $X_n$-distinguished $g \in T_n(k)$ always has a term that 
    just involves the variable $X_n$. Indeed, if 
    $g = \sum_{\nu=0}^\infty g_\nu X^\nu_n \in T_n(k)$ is distinguished of order 
    $s$, then $g_s$, being a unit, has a nonzero constant term. Then $g$ has a term
    of the form $a_sX^s_n$, where $a_s \in k^\times$ is the constant term of $g_s$.
    
    \item \label{T1-distinguished} Every nonzero element of $T_1(k) = k\{ X \}$ 
    is $X$-distinguished of order equal to the value of the element under the
    associated Euclidean function that gives $T_1(k)$ the structure of a Euclidean domain
    (see Theorem \ref{thm:tateprops}$(\ref{thm:T_1euclid})$).
\end{enumerate}
\end{remark}

\par The next result can be thought of as the Tate algebra
analogue of a technique used in the proof of Noether normalization
for finite type algebras over a field that makes polynomials 
essentially monic in one of the variables upon applying a suitable ring automorphism. 
In fact, this result is
used to prove
the rigid-analytic analogue of Noether normalization
for quotients of Tate algebras.

\begin{lemma}[see {\cite[Lem.\ 2.2/7]{Bos14}}]
\label{lem:rigid-distinguished}
Given finitely many nonzero elements $g_1, g_2,\ldots, g_r \in T_n(k)$,
there is a $k$-algebra automorphism (automatically continuous)
\[
  \sigma\colon T_n(k) \longrightarrow T_n(k), \qquad
  X_i \longmapsto \begin{cases}
    X_i + X_n^{\alpha_i} & \text{for}\ i < n\\
    X_n & \text{for}\ i = n
  \end{cases}
\]
for some $\alpha_1,\alpha_2,\ldots,\alpha_{n-1} \in
\ZZ_{\ge0}$ such that the elements $\sigma(g_1),\sigma(g_2),\ldots,\sigma(g_r)$
are $X_n$-distinguished.
\end{lemma}

\subsection{Continuous maps, spherically complete fields, and the Hahn--Banach
extension property}
\label{subsection:NA-functional-analysis} We begin by giving
alternate characterizations of continuous maps of normed spaces. Recall that
by Convention \ref{conv:nontrivial}, all non-Archimedean fields are complete 
and non-trivially valued. Given a normed space $(E,\norm{})$,
we will say that a subset $S \subseteq E$ is \textsl{bounded} if and only if there
exists some $a \in \RR_{\geq 0}$ such that $S$ is contained in the closed ball $B_a(0)$
of radius $a$ centered at $0 \in E$.

\begin{lemma}
\label{lem:continuous}
Let $(k,\abs{})$ be a non-Archimedean field and $(E,\norm{}_E)$, $(F,\norm{}_F)$
be normed spaces. Then, for a $k$-linear map $f\colon E \rightarrow F$, the 
following are equivalent:
\begin{enumerate}[label=$(\roman*)$,ref=\roman*]
  \item\label{lem:continuouscont} $f$ is continuous.
  \item\label{lem:continuousnull} $f$ maps null sequences to null sequences.
  \item\label{lem:null-to-bounded} $f$ maps null sequences to bounded sequences.
  \item\label{lem:bounded} $f$ maps bounded sets to bounded sets.
  \item\label{lem:one-bounded} There exists $a, b \in \RR_{> 0}$ such that 
  $f(B_a(0)) \subseteq B_b(0)$.
  \item\label{lem:continuousconstb} There exists $B \in \RR_{> 0}$ such that for
    all $x \in E$, we have 
	$\norm{f(x)}_F \leq B\cdot\norm{x}_E$.
\end{enumerate}
\end{lemma}

\begin{proof}
Since a null sequence converges
to $0$, $(\ref{lem:continuouscont}) \Rightarrow
(\ref{lem:continuousnull})$ follows by the continuity of $f$ and the fact that
$f$ maps $0$ to $0$. $(\ref{lem:continuousnull})\Rightarrow
(\ref{lem:null-to-bounded})$ holds because null sequences are bounded.

\par For $(\ref{lem:null-to-bounded})\Rightarrow(\ref{lem:bounded})$,
assume for contradiction that there exists some $a \in \RR_{> 0}$ such
that $f(B_a(0))$ is unbounded. Then, there exists a sequence 
$(x_n)_n \subseteq B_a(0)$ and a sequence $(c_n)_n \subseteq k^\times$ such that 
\[\sqrt{\norm{f(x_n)}_F} \ge \abs{c_n}\]
for all $n > 0$,
and $\abs{c_n} \rightarrow \infty$ as $n \rightarrow \infty$. Here we use the 
non-triviality of $\abs{k^\times}$ to get the sequence $(c_n)_n$ with
said properties.
Since $(x_n)_n$ is bounded, $(c^{-1}_nx_n)_n$ is a null sequence
whose image is an unbounded sequence, because 
\[
\norm{f(c^{-1}_nx_n)}_F = \abs{c^{-1}_n}\cdot\norm{f(x_n)}_F \geq \sqrt{\norm{f(x_n)}_F}.
\] 
This contradicts $(\ref{lem:null-to-bounded})$, completing the proof of
$(\ref{lem:null-to-bounded})\Rightarrow (\ref{lem:bounded})$.

\par The proof of $(\ref{lem:bounded}) \Rightarrow (\ref{lem:one-bounded})$
is clear. 
For $(\ref{lem:one-bounded}) \Rightarrow (\ref{lem:continuousconstb})$, 
by non-triviality of
$\abs{k^\times}$, we may choose a nonzero $c \in k$ such that
$0 < \abs{c} < 1$, $\abs{c} \leq a$, and consequently,
$f(B_{\abs{c}}(0)) \subseteq B_b(0)$.
Let $x \in E$ and choose $m \in \mathbf{Z}$
such that
\begin{equation}
\label{eq:clever-m}
\abs{c}^{m+2} < \norm{x}_E \leq \abs{c}^{m+1}. 
\end{equation}
We have $\norm{c^{-m}x}_E \leq \abs{c}$, and so,
\[
  \norm{f(x)}_F = \abs{c}^m\,\norm{f(c^{-m}x)}_F \leq \abs{c}^m\,b \leq
  \abs{c}^{-2}\,b\,\norm{x}_E,
\]
where the middle inequality follows because $f(B_{\abs{c}}(0)) \subseteq B_b(0)$, and the final
inequality follows from \eqref{eq:clever-m}. Then taking $B \coloneqq
\abs{c}^{-2}\,b$, we get
$(\ref{lem:one-bounded}) \Rightarrow (\ref{lem:continuousconstb})$. 

\par Finally, $(\ref{lem:continuousconstb}) \Rightarrow (\ref{lem:continuouscont})$ follows 
by using
the $\epsilon$-$\delta$ definition of continuity. This finishes the proof of the Lemma.
\end{proof}

\par Thus, by Lemma \ref{lem:continuous}, 
for a continuous linear map $f\colon (E,\norm{}_E) \rightarrow (F,\norm{}_F)$ 
of normed spaces,
\[
  \sup_{x \ne 0}\biggl\{\frac{\norm{f(x)}_F}{\norm{x}_E}\biggr\}
\]
is finite. In other words, all continuous maps of normed spaces
are \textsl{bounded continuous}.

\par We next introduce the Hahn--Banach extension property over non-Archimedean fields. 
The corresponding extension property over $\RR$ or $\mathbf{C}$ is one of the
most important results in functional analysis.

\begin{definition}[see {\cite[p.\ 170]{PGS10}}]
\label{def:Hahn-Banach}
A normed space $(E,\norm{})$ over a non-Archimedean field $(k,\abs{})$
satisfies the \textsl{$(1+\epsilon)$-Hahn--Banach extension property} if for every 
subspace $D$ of $E$, for every $\epsilon > 0$, and for every linear functional 
$f\colon D \rightarrow k$
such that $\abs{f(x)} \leq \norm{x}$ for all $x \in D$, there exists
a linear functional $\widetilde{f}\colon E \rightarrow k$ extending $f$ 
such that for
all $x \in E$, we have
\[
  \abs{\widetilde{f}(x)} \leq (1+\epsilon)\norm{x}.
\]
We say
$E$ satisfies the \textsl{Hahn--Banach extension property} if $\epsilon$
can be chosen to be $0$ for the extension $\widetilde{f}$, that is,
if there exists an extension $\widetilde{f}$ such that $\abs{\widetilde{f}(x)} 
\leq \norm{x}$ for all $x \in E$.
\end{definition}

\par The terminology ``$(1+\epsilon)$-Hahn--Banach extension property'' is used in passing in
\cite{PGS10}, but it is convenient.
Note that both $f$ and its extension $\widetilde{f}$ are continuous by Lemma
\ref{lem:continuous}. In
particular, the $(1+\epsilon)$-Hahn--Banach extension property guarantees that $E$ has
many continuous linear functionals. We are particularly interested in
conditions on the non-Archimedean field $k$ which guarantee that every normed
space over $k$ satisfies the $(1+\epsilon)$-Hahn--Banach extension property
or the usual Hahn--Banach extension property. 
This leads to the notion of a spherically complete
field, which should be viewed as a generalization of a (complete) non-Archimedean field
with a discrete value group.

\begin{definition}
\label{def:spherically-complete}
A non-Archimedean field $(k,\abs{})$ is \textsl{spherically complete}
if, for every decreasing sequence of closed disks 
\[
  D_1 \supseteq D_2 \supseteq D_3 \supseteq \cdots
\]
the intersection $\bigcap_n D_n$ is non-empty.
\end{definition}

\begin{remark}
\label{rem:spherical-comp-facts}
  \leavevmode
\begin{enumerate}[label=$(\roman*)$,ref=\roman*]
  \item The defining property of spherical completeness implies completeness,
	even though our convention is that non-Archimedean fields are complete
  (see \cite[p.\ 24]{vR78}).
	
\item\label{rem:spherical-comp-discrete}
  If $\abs{k^\times} \cong \ZZ$, then $k$ is spherically complete. In other
	words, the fraction field of a complete discrete valuation ring is spherically
  complete in the topology induced by the valuation (see \cite[Cor.\ 2.4]{vR78}).
	
  \item\label{rem:spherical-comp-exists}
    There exist non-Archimedean fields that are not spherically complete.
    For example,
	$\mathbf{C}_p$, the completion of the algebraic closure of $\mathbf{Q}_p$,
  is not spherically complete \cite[Cor.\ 20.6]{Sch06}. However, every 
  non-Archimedean field admits
  an embedding into a spherically complete field \cite[Satz 24]{Kru32} (see also
  \cite[Thm.\ 4.49]{vR78}).
	
	\item\label{maximal-extensions} A non-Archimedean field $(k,\abs{})$ is 
	spherically complete
    if and only if $k$ admits no proper immediate extensions \cite[Thm.\
    4]{Kap42} (see also \cite[Thm.\ 4.47]{vR78}).
	Recall that we say that an extension of real-valued fields
	$(k,\abs{}_k) \hookrightarrow (\ell, \abs{}_\ell)$ (i.e.\ an extension of
  fields such that $\abs{}_\ell$
	restricted to $k$ equals $\abs{}_k$) is an \textsl{immediate
	extension} if $\abs{k^\times}_k = \abs{\ell^\times}_\ell$ and the 
	induced map on residue fields 
	$k^\circ/k^{\circ\circ} \hookrightarrow \ell^\circ/\ell^{\circ\circ}$
	is an isomorphism.
\end{enumerate}
\end{remark}

\par Every normed space over a spherically complete field satisfies the Hahn--Banach 
extension property. This, and some related results, are summarized below.

\begin{theorem}
\label{thm:Hahn-Banach}
Let $(E,\norm{})$ be a normed space over a non-Archimedean field $(k,\abs{})$.
\begin{enumerate}[label=$(\roman*)$,ref=\roman*]
  \item\label{thm:Hahn-Banachsph}
    If $k$ is spherically complete (in particular, if $k$ is discretely valued),
    then $E$ satisfies the Hahn--Banach
extension property.

\item\label{thm:Hahn-Banachcountable}
  If $E$ has a dense subspace $V$ which has a countable basis over $k$, then
$E$ satisfies the $(1+\epsilon)$-Hahn--Banach extension property.

\item\label{thm:Hahn-Banachpolar}
  Assume that $\abs{k^\times}$ is not discrete, and
  suppose the norm $\norm{}$ on $E$ is \textsl{polar} in the sense that for every
  \[
    x \notin E^\circ \coloneqq \Set{x \in E \given \norm{x} \leq 1},
  \]
  there exists a linear functional $f\colon E \rightarrow k$ such that
$\abs{f(E^\circ)} \leq 1$ and $\abs{f(x)} > 1$.
For every finite-dimensional 
subspace $D$ of $E$,
if $f\colon D \rightarrow k$
is a linear map that satisfies $\abs{f(x)} \leq \norm{x}$ 
for all $x \in D$, then
for all $\epsilon > 0$, there exists an extension 
$\widetilde{f}\colon E \rightarrow k$ of $f$ such that 
$\abs{\widetilde{f}(x)} \leq (1+\epsilon)\norm{x}$ for all 
$x \in E$.
\end{enumerate}
\end{theorem}

\begin{proof}
$(\ref{thm:Hahn-Banachsph})$ follows from \cite[Thm.\ 4.1.1]{PGS10},
$(\ref{thm:Hahn-Banachcountable})$ from \cite[Thm.\ 4.2.4]{PGS10},
and $(\ref{thm:Hahn-Banachpolar})$ from \cite[Thm.\
4.4.5]{PGS10}.
In $(\ref{thm:Hahn-Banachsph})$, the discretely valued case
follows from the spherically complete case by Remark
\ref{rem:spherical-comp-facts}$(\ref{rem:spherical-comp-discrete})$.
\end{proof}

\begin{remark}
In \cite{PGS10}, norms are assumed to be
\textsl{solid} in the sense of \cite[Def.\ 3.1.1]{PGS10} (see \cite[p.\
93]{PGS10}).
However, the solidity assumption is not used in the results \cite[Thms.\ 4.1.1 and
4.2.4]{PGS10} cited for the proofs of $(\ref{thm:Hahn-Banachsph})$ and
$(\ref{thm:Hahn-Banachcountable})$ above. Moreover, solidity is automatic
if $\abs{k^\times}$ is dense in $\RR_{\geq 0}$, or equivalently, when $\abs{k^\times}$
is not discrete. For $(\ref{thm:Hahn-Banachpolar})$, the notion of polarity was defined by
Schikhof as a common generalization of the situations in
$(\ref{thm:Hahn-Banachsph})$ and $(\ref{thm:Hahn-Banachcountable})$
(see \cite[Def.\ 4.4.1, Lem.\ 4.4.4, and Thm.\ 4.4.3]{PGS10}).
\end{remark}

\par Theorem \ref{thm:Hahn-Banach} raises the natural question of whether
one can construct Banach spaces over a non-Archimedean field whose
continuous dual space is trivial. The next result shows that such spaces
exist for every non-Archimedean field that is not spherically complete 
because any non-Archimedean field embeds into one that is spherically
complete by Remark
\ref{rem:spherical-comp-facts}$(\ref{rem:spherical-comp-exists})$.

\begin{proposition}
\label{prop:trivial-dual-space}
Consider an extension of non-Archimedean fields $(k,\abs{}_k) \hookrightarrow 
(\ell, \abs{}_\ell)$ such that $\ell$ is spherically complete but $k$ is not. Then there
are no nonzero continuous $k$-linear maps $\ell \rightarrow k$.
\end{proposition}

\begin{proof}
The Proposition follows by setting $E$ to be $\ell$ in the statement of
\cite[Cor.\ 4.3]{vR78} or in the proof of \cite[Thm.\ 2]{vdPvT67}.
\end{proof}

\begin{remark}
For the curious reader who prefers a more explicit example of a Banach space over
a non-spherically complete non-Archimedean field $(k,\abs{})$ that admits no nonzero
continuous functionals, note that if
$\ell^\infty$ is the $k$-Banach
space of sequences of elements of $k$ with bounded norms and $c_0$ is the closed
subspace of null sequences, then $k$ is spherically complete if and only if 
$\ell^\infty/c_0$ admits a nonzero continuous linear functional 
\cite[Cor. 4.1.13]{PGS10}.
\end{remark}

\section{Frobenius splitting of some Tate algebras}\label{sec:refined-mainthm}
We begin by giving a topological characterization of Frobenius splitting
of Tate algebras over non-Archimedean fields of prime characteristic, which
is a stronger version of Theorem \hyperref[thm:refined-mainthmintro]{B}.
\begin{theorem}
\label{thm:refined-mainthm}
  Let $(k,\abs{})$ be a non-Archimedean field of
  characteristic $p > 0$.
  The following are equivalent:
  \begin{enumerate}[label=$(\roman*)$,ref=\roman*]
    \item\label{thm:tatefsplitnarb}
      $T_n(k)$ is Frobenius split for each integer $n > 0$.
     \item\label{thm:tatesolidarbn}
     $T_n(k)$ has a nonzero $p^{-1}$-linear map for each integer $n > 0$.
    \item\label{thm:tatefsolidsomen}
      There exists an integer $n > 0$ for which $T_n(k)$ has a nonzero $p^{-1}$-linear map.
    \item\label{thm:tatenontrivial}
      $T_1(k)$ has a nonzero $p^{-1}$-linear map.
    \item\label{thm:tatefsplitn1}
      $T_1(k)$ is Frobenius split.
    \item\label{thm:tatefsplitklinmap}
      There exists a nonzero continuous $k$-linear map $f\colon
      F_{k*}k \to k$.
  \end{enumerate}
\end{theorem}

\begin{proof}
  $(\ref{thm:tatefsplitnarb}) \Rightarrow
  (\ref{thm:tatesolidarbn})$ is clear since Frobenius splittings
  are nonzero $p^{-1}$-linear maps. Similarly, so is $(\ref{thm:tatesolidarbn})
  \Rightarrow (\ref{thm:tatefsolidsomen})$.
  
  \par For $(\ref{thm:tatefsolidsomen})
  \Rightarrow (\ref{thm:tatenontrivial})$, we may assume $n > 1$. 
  Let $\Phi\colon F_{T_n*}T_n \to T_n$ be a nonzero $T_n$-linear map, and choose $a
  \in F_{T_n*}T_n$ such that $\Phi(a) \ne 0$. 
  By replacing $\Phi$ with $\Phi \circ F_{T_n*}(- \cdot a)$, we may assume that
  $\Phi$ is such that $\Phi(1) \ne 0$.
  We now modify $\Phi$ to create a $T_n$-linear map
  $\widetilde{\Phi}\colon F_{T_n*}T_n \to T_n$ such that $1$ maps to an element
  not contained in the ideal $(X_1,X_2,\ldots,X_{n-1})$.
  By Lemma \ref{lem:rigid-distinguished}, there exists an automorphism
  $\sigma\colon T_n \to T_n$ such that $g = \sigma(\Phi(1))$ is
  $X_n$-distinguished in the sense of Definition \ref{def:Weierstrass-dist}.
  In particular, $g \notin (X_1,X_2,\dots,X_{n-1})$ by Remark \ref{rem:units}$(\ref{terms})$.
  Now consider the composition
  \[
    \widetilde{\Phi}\colon
    F_{T_n*}T_n \xrightarrow{F_{T_n*}(\sigma^{-1})} F_{T_n*}T_n
    \overset{\Phi}{\longrightarrow} T_n \overset{\sigma}{\longrightarrow} T_n.
  \]
  We claim that $\widetilde{\Phi}$ is $T_n$-linear, or said differently,
	that  $\widetilde{\Phi}$
  defines a $p^{-1}$-linear map on $T_n$. We will exploit the fact that $\sigma$
  is a ring automorphism to see this.
  Let $h \in T_n$ and $f \in F_{T_n*}T_n$.
  We then have
  \begin{align*}
    \widetilde{\Phi}(h \cdot f) &= \widetilde{\Phi}(h^pf)\\
    &= (\sigma \circ \Phi) \bigl(\sigma^{-1}(h^pf)\bigr)\\
    &= (\sigma \circ \Phi) \bigl(\sigma^{-1}(h)^p\,\sigma^{-1}(f))\bigr)\\
    &= \sigma\bigl(\sigma^{-1}(h)\cdot\Phi\bigl(\sigma^{-1}(f)\bigr)\bigr)\\
    &= h\cdot\sigma\bigl(\Phi\bigl(\sigma^{-1}(f)\bigr)\bigr)\\
    &= h\cdot\widetilde{\Phi}(f)
  \end{align*}
  as desired.
  Since $\sigma^{-1}$ maps $1$ to $1$, we see that
  $\widetilde{\Phi}$ maps $1$ to the $X_n$-distinguished element $g$.
  Finally, consider the composition
  \[
    F_{k\{ X_n \}*}k\{ X_n \} \hooklongrightarrow
    F_{T_n*}T_n \overset{\widetilde{\Phi}}{\longrightarrow}
    T_n \overset{\pi}{\longtwoheadrightarrow} k\{ X_n \},
  \]
  where $\pi$ is the quotient map sending $X_1,X_2,\ldots,X_{n-1}$ to $0$.
  The composition defines a $p^{-1}$-linear map on $k\{ X_n \}$, and sends
  $1$ to a nonzero element in $k\{ X_n \}$ because 
  $g \notin (X_1,X_2,\ldots,X_{n-1}) = \ker(\pi)$.

  \par For $(\ref{thm:tatenontrivial}) \Rightarrow (\ref{thm:tatefsplitn1})$, we will 
  use the fact that $T_1$ is a Euclidean domain 
  (Theorem \ref{thm:tateprops}$(\ref{thm:T_1euclid})$), hence in particular a principal
  ideal domain (PID). Let $K \coloneqq \Frac(T_1)$ and let
  \[
  \Phi\colon F_{T_1*}T_1 \longrightarrow T_1
  \]
  be a nonzero $T_1$-linear map. Since $T_1$ is a PID, there exists a nonzero $a \in T_1$
  such that $\im(\Phi) = aT_1$. Then by construction, the composition
  \[
  F_{T_1*}T_1 \overset{\Phi}{\longrightarrow} T_1 \hooklongrightarrow K \xrightarrow{-\cdot a^{-1}} K
  \]
  is a $T_1$-linear map whose image is $T_1$. Therefore,
  restricting the codomain of this composition to $T_1$ gives us a surjective $T_1$-linear map
  $\widetilde{\Phi}\colon F_{T_1*}T_1 \twoheadrightarrow T_1$.
  Let $x \in F_{T_1*}T_1$ such that 
  $\widetilde{\Phi}(x) = 1$.
  Then, the composition
  \[
    F_{T_1*}T_1 \xrightarrow{F_{T_1*}(-\cdot x)} F_{T_1*}T_1
  \overset{\widetilde{\Phi}}{\longrightarrow} T_1
  \]
  maps $1 \in F_{T_1*}T_1$ to $1 \in T_1$, and hence $T_1$ is Frobenius split.
  
  \par We now show $(\ref{thm:tatefsplitklinmap}) \Rightarrow
  (\ref{thm:tatefsplitnarb})$.
  Let $c \in F_{k*}k$ be such that $f(c) = b \ne 0$.
  Then, the composition
  \[
    \phi\colon F_{k*}k
    \xrightarrow{F_{k*}(-\cdot c)} F_{k*}k 
    \overset{f}{\longrightarrow} k
    \xrightarrow{-\cdot b^{-1}} k
  \]
  is a continuous $k$-linear splitting of the Frobenius map
  $F_k\colon k \to F_{k*}k$.
  Now on $F_{T_n*}T_n$, we consider the map
  \[
    \begin{tikzcd}[column sep=1.475em,row sep=0]
      \mathllap{\Phi\colon} F_{T_n*}T_n \rar & T_n\\
      \displaystyle\sum_{\nu \in \ZZ_{\ge0}^n} a_\nu X^\nu 
      \rar[mapsto]
      & \displaystyle\sum_{\nu \in p \cdot \ZZ_{\ge0}^n} \phi(a_\nu)
      X^{\nu/p}
    \end{tikzcd}
  \]
  of $T_n$-modules, where $p \cdot \ZZ_{\ge0}^n$ denotes the submonoid of
  $\ZZ_{\ge0}^n$ where each coordinate is divisible by $p$, and for $\nu \in p
  \cdot \ZZ_{\ge0}^n$, the multi-index $\nu/p$ is obtained by dividing every
  coordinate by $p$.
  The map indeed defines a map to $T_n$ since if $\abs{a_\nu} \to 0$,
  then $\abs{\phi(a_\nu)} \to 0$ by continuity of $\phi$ (see Lemma \ref{lem:continuous}).
  Since $\Phi$ maps $1 \in F_{T_n*}T_n$ to $1 \in T_n$, we get a Frobenius
  splitting of $T_n$.
  
  \par It remains to show $(\ref{thm:tatefsplitn1}) \Rightarrow
  (\ref{thm:tatefsplitklinmap})$.
  Let $\Phi\colon F_{k\{ X \}*}k\{ X \} \to k\{ X
  \}$ be a Frobenius splitting of $T_1 = k\{ X \}$, and consider
  the composition
  \[
    f\colon F_{k*}k \hooklongrightarrow F_{k\{ X \}*}k\{ X \}
    \overset{\Phi}{\longrightarrow} k\{ X \} \longtwoheadrightarrow
    \frac{k\{ X \}}{(X)} \overset{\sim}{\longrightarrow} k.
  \]
  Note that $f$ is a nonzero $k$-linear map since $f$ maps $1 \in
  F_{k*}k$ to $1 \in k$. Assume for
contradiction that $f$ is not continuous, and choose by Lemma \ref{lem:continuous}
a sequence $(a_i)_{i \in \mathbf{Z}_{\geq 0}} \subseteq F_{k*}k$ such that
$\abs{a_i} \rightarrow 0$ as $i \rightarrow \infty$ and such that denoting
\[
\Phi(a_i) = \sum_{j=0}^\infty b_{i,j}X^j,
\]
we have
\[
  \abs{b_{i,0}} \geq i!
\]
for all $i$.\footnote{We assume more than we need here so that
our argument easily adapts to the setting of Theorem \ref{prop:local-analogue}.}
Note that $f(a_i) = b_{i,0}$.
Using this sequence $(a_i)_i$,
we construct a restricted power series in $k\{ X \}$ whose image
under $\Phi$ does not land in $k\{ X \}$.
Let $m_0 \coloneqq 0$, and for all $i \geq 1$, inductively 
choose $m_i \gg m_{i-1}$ such that
\[
\max_{0 \le r \le i-1}\Bigl\{\abs[\big]{b_{m_r,i-r}}\Bigr\} <
    \abs[\big]{b_{m_i,0}}.
\]
Note that such $m_i$ exist because $\abs{b_{i,0}} \rightarrow \infty$ as $i \rightarrow \infty$.
We then have
\begin{equation}\label{eq:howwechoosemi}
    \abs[\Bigg]{\sum_{r=0}^{i-1} b_{m_r,i-r}} \le
    \max_{0 \le r \le i-1}\Bigl\{\abs[\big]{b_{m_r,i-r}}\Bigr\} <
    \abs[\big]{b_{m_i,0}}
  \end{equation}
by the non-Archimedean triangle inequality.
Now consider the restricted power series 
\[\sum_{r=0}^\infty a_{m_r}X^{rp} \in k\{X\}.\]
Applying the Frobenius splitting $\Phi$ to this power series we see that for all $i \in \ZZ_{\ge 0}$, 
  \begin{align}
    \Phi\Biggl(\sum_{r=0}^\infty a_{m_r} X^{rp}\Biggr)
    &=\Phi\Biggl(\sum_{r=0}^i a_{m_r} X^{rp}\Biggr) + 
    \Phi\Biggl(\sum_{r=i+1}^\infty a_{m_r} X^{rp}\Biggr)\nonumber\\
    &= \sum_{r=0}^i\Biggl(X^r\sum_{j=0}^{\infty}b_{m_r,j}X^j\Biggr) + X^{i+1}
    \Phi\Biggl(\sum_{r=i+1}^\infty a_{m_r} X^{(r-i-1)p}\Biggr).
    \label{eq:imageofbadseries}
  \end{align}
 Note that the second term in \eqref{eq:imageofbadseries} is divisible by
  $X^{i+1}$.
  Thus, it does not contribute to the coefficient of $X^{i}$ in 
  $\Phi(\sum_{r=0}^\infty a_{m_r} X^{rp})$.
  For every $i \in
  \ZZ_{\ge 0}$, we then have
\begin{align*}
    \MoveEqLeft[3]\abs[\Bigg]{\text{coefficient of $X^{i}$ in
    $\Phi\Biggl(\sum_{r=0}^\infty a_{m_r} X^{rp}\Biggr)$}}\\
    &= \abs[\Bigg]{\sum_{r=0}^i b_{m_r,i-r}}
    = \abs[\Bigg]{b_{m_i,0} + \sum_{r=0}^{i-1} b_{m_r,i-r}} 
    =\abs[\big]{b_{m_i,0}} \ge {m_i}!,
  \end{align*}
where the penultimate equality follows by Lemma \ref{lem:max-norm}
and by \eqref{eq:howwechoosemi}.
Therefore $\Phi(\sum_{r=0}^\infty a_{m_r} X^{rp})$ cannot be a restricted power series,
contradicting the assumption that $\Phi$ maps into $T_1$. Thus, $f$ must be continuous.
\end{proof}

\par This topological characterization shows that Tate algebras are Frobenius split
in many commonly occurring cases.

\begin{customcor}{D}
\label{cor:F-split-usually}
  Let $(k,\abs{})$ be a complete non-Archimedean field of
  characteristic $p > 0$.
For each $n > 0$, the Tate algebra $T_n(k)$ is Frobenius split in the
following cases:
\begin{enumerate}[label=$(\roman*)$,ref=\roman*]
	\item\label{cor:F-split-usually-spher} $(k,\abs{})$ is spherically complete.
	\item\label{cor:F-split-usually-count} $k^{1/p}$ has a dense $k$-subspace $V$ that has a countable $k$-basis, hence in particular
	if $[k^{1/p}:k] < \infty$.
	\item\label{cor:F-split-polar} $\abs{k^\times}$ is not discrete, and the norm on $k^{1/p}$ is polar.
\end{enumerate}
\end{customcor}
We note that $(\ref{cor:F-split-usually-spher})$ and
$(\ref{cor:F-split-usually-count})$ arose out of conversations with Eric Canton
and Matthew Stevenson.
\begin{proof}
In each of the above cases, Theorem \ref{thm:Hahn-Banach} implies
that the identity map $\id_k\colon k \rightarrow k$ can be extended
to a continuous $k$-linear map $k^{1/p} \rightarrow k$. We are then
done by Theorem \ref{thm:refined-mainthm}.
\end{proof}

\section{A local construction}\label{sec:local-construction}
We now show how the topological characterization
of Frobenius splittings of Tate algebras can be extended to a similar local
construction involving convergent power series rings.
This will in turn yield local examples of excellent regular
rings that are not Frobenius split in Section \ref{sec:troublesome-field}.

\begin{definition}[{see \cite[pp.\ 190--191]{Nag62}}]
\label{def:convergent-series}
  Let $(k,\abs{})$ be a non-Archimedean field.
  For every positive integer $n > 0$, the \textsl{convergent power series ring}
  in $n$ indeterminates over $k$ is the $k$-subalgebra
  \[
    K_n(k) \coloneqq k\langle X_1,X_2,\ldots,X_n \rangle
    \coloneqq
    \Set[\Bigg]{\sum_{\nu \in \ZZ_{\ge0}^n} a_\nu X^{\nu} \given
      \begin{array}{@{}c@{}}
        a_\nu \in k\ \text{and there exist}\ r_1,r_2,\ldots,r_n \in
        \RR_{>0}\ \text{and}\ M \in \RR_{>0}\\
        \text{such that}\ \abs{a_\nu}\,r^{\nu_1}_1\ldots r^{\nu_n}_n \leq M\
        \text{for all}\ \nu \in \ZZ_{\ge0}^n
      \end{array}}
  \]
  of the formal power
  series $k\llbracket X_1,X_2,\ldots,X_n \rrbracket$ in $n$ indeterminates over $k$.\footnote{Following \cite{GR71},
  the letter $K$ is used instead of the letter $C$ because the German word for
  ``convergent'' is ``konvergent.''}
  For $n = 1$, we will denote the indeterminate in $K_1(k)$ by just $X$ instead
  of $X_1$.
\end{definition}

\begin{remark}
\label{rem:exp}
It is clear that $T_n(k)$ is a subring of $K_n(k)$ for each $n > 0$. For $n=1$,
an element of $K_1(k)$ is a power series $\sum_i a_i X^i$ for which there exists
real numbers $r, M > 0$ such that $\abs{a_i} \leq Mr^{-i}$ for all $i \in \ZZ_{\geq 0}$.
We can always assume $0 < r < 1$, and so, $\sum_i a_i X^i$ is convergent
precisely when the norms of its coefficients can be bounded by some exponential function.
\end{remark}

\par The main properties of $K_n(k)$ are summarized below.

\begin{theorem}
\label{prop:properties-K_n}
  Let $(k,\abs{})$ be a non-Archimedean field, and let $n$ be a
  positive integer.
  Then, the convergent power series ring $K_n(k)$ satisfies the following
  properties:
\begin{enumerate}[label=$(\roman*)$,ref=\roman*]
  \item\label{conv-local} $K_n(k)$ is a Noetherian local ring of Krull dimension
    $n$ whose maximal ideal is $(X_1,X_2,\dots,X_n)$.
  \item\label{conv-regular} $K_n(k)$ is regular.
  \item\label{conv-Hens} $K_n(k)$ is Henselian.
	\item\label{conv-exc} $K_n(k)$ is excellent.
\end{enumerate}
\end{theorem}

\begin{proof}[Indication of proof]
$(\ref{conv-local})$--$(\ref{conv-Hens})$ are proved in 
\cite[Thm.\ 45.5]{Nag62}.
For $(\ref{conv-exc})$, in the proof that $T_n(k)$ is excellent, Kiehl observes
that one can adapt the proof for $T_n(k)$ to show that $K_n(k)$ is also
excellent \cite[p.\ 89]{Kie69} (see also \cite[Thm.\ 1.1.3]{Con99}).
\end{proof}

\par With these preliminaries, one can now adapt the proof of Theorem \ref{thm:refined-mainthm}
to obtain an analogous set of results for the local ring $K_n(k)$.

\begin{theorem}
\label{prop:local-analogue}
Let $(k,\abs{})$ be a non-Archimedean field of
  characteristic $p > 0$.
  The following are equivalent:
  \begin{enumerate}[label=$(\roman*)$,ref=\roman*]
    \item\label{thm:convfsplitnarb}
      $K_n(k)$ is Frobenius split for each integer $n > 0$.
    \item\label{thm:convfsolidalln}
      $K_n(k)$ has a nonzero $p^{-1}$-linear map for each integer $n > 0$.
    \item\label{thm:convfsolidsomen}
    There exists an integer $n > 0$ for which $K_n(k)$ has a nonzero $p^{-1}$-linear map.
    \item\label{thm:convnontrivial}
      $K_1(k)$ has a nonzero $p^{-1}$-linear map.
    \item\label{thm:convfsplitn1}
      $K_1(k)$ is Frobenius split.
    \item\label{thm:convfsplitklinmap}
      There exists a nonzero continuous $k$-linear map $f\colon
      F_{k*}k \to k$.
  \end{enumerate}
\end{theorem}

\begin{proof}
$(\ref{thm:convfsplitnarb})\Rightarrow(\ref{thm:convfsolidalln})$ 
and $(\ref{thm:convfsolidalln})\Rightarrow
(\ref{thm:convfsolidsomen})$ are clear.

\par For the proof of $(\ref{thm:convfsolidsomen}) \Rightarrow
(\ref{thm:convnontrivial})$, we adapt the proof in Theorem
\ref{thm:refined-mainthm} as follows.
Instead of Lemma \ref{lem:rigid-distinguished}, one uses \cite[Kap.\ I,
Folgerung zu Satz 4.3]{GR71} to find an automorphism $\sigma\colon K_n \to K_n$
such that $\sigma(\Phi(1))$ is $X_n$-distinguished in the sense of \cite[Kap.\
I, \S4.1]{GR71}.\footnote{While the definition of $K_n$ in \cite[Kap.\ I,
\S3.1]{GR71} differs from that in Definition \ref{def:convergent-series}, they
are equivalent by Abel's lemma for convergence of power series \cite[Kap.\ I,
Satz 1.2]{GR71}.}
The definition of $X_n$-distinguished still implies that $g \notin
(X_1,X_2,\ldots,X_{n-1})$, and the rest of the proof is the same as in Theorem
\ref{thm:refined-mainthm}.

\par The proof of $(\ref{thm:convnontrivial})\Rightarrow
(\ref{thm:convfsplitn1})$ follows by adapting the corresponding proof
in Theorem \ref{thm:refined-mainthm} using
the observation that $K_1$, being a regular local ring of dimension $1$,
is a PID (Theorem \ref{prop:properties-K_n}).

\par For $(\ref{thm:convfsplitklinmap})\Rightarrow(\ref{thm:convfsplitnarb})$,
we may assume without loss of generality that $f$ maps $1 \in F_{k*}k$
to $1 \in k$ as in Theorem
\ref{thm:refined-mainthm}. By Lemma
\ref{lem:continuous}$(\ref{lem:continuousconstb})$, choose a positive real 
number $B$ such that
for all $x \in F_{k*}k$, we have $\abs{f(x)} \leq B \cdot \abs{x}$. 
Now on $F_{K_n*}K_n$, we consider the map 
  \[
    \begin{tikzcd}[column sep=1.475em,row sep=0]
      \mathllap{\Phi\colon} F_{K_n*}K_n \rar & K_n\\
      \displaystyle\sum_{\nu \in \ZZ_{\ge0}^n} a_\nu X^\nu 
      \rar[mapsto]
      & \displaystyle\sum_{\nu \in p \cdot \ZZ_{\ge0}^n} f(a_\nu)
      X^{\nu/p},
    \end{tikzcd}
  \]
  of $K_n$-modules,
  which we claim indeed maps to $K_n$. This is because if $r_1,r_2,\ldots,r_n$
  and $M$
are positive real numbers such that for every $\nu \in \mathbf{Z}^n_{\geq 0}$ we 
have $\abs{a_\nu}r^{\nu_1}_1\cdots r^{\nu_n}_n \leq M$, then for every
$\nu \in p\cdot \mathbf{Z}^n_{\geq 0}$, we have
\[
  \abs{f(a_\nu)}\,(r_1^p)^{\nu_1/p}(r_2^p)^{\nu_2/p}\cdots (r_n^p)^{\nu_n/p} 
  \leq B\cdot\abs{a_\nu}\,r_1^{\nu_1}r_2^{\nu_2}\cdots r_n^{\nu_n} \leq B\cdot M.
\]
Said differently, the defining condition of a convergent power series
can be checked for $\sum_{\nu \in p \cdot \ZZ_{\ge0}^n} f(a_\nu)
      X^{\nu/p}$
upon replacing $r_1,r_2,\dots,r_n$
by $r^p_1,r^p_2,\dots,r^p_n$ and $M$ by $B \cdot M$.
Since $\Phi$ maps $1 \in F_{K_n*}K_n$ to $1 \in K_n$, we get
a Frobenius splitting of $K_n$. 

The proof of $(\ref{thm:convfsplitn1})\Rightarrow(\ref{thm:convfsplitklinmap})$
follows from the proof of the same implication in Theorem \ref{thm:refined-mainthm}
by replacing $T_1$ by $K_1$ and the phrase ``restricted power series'' by the phrase
``convergent power series'' everywhere. This is because in Theorem \ref{thm:refined-mainthm},
assuming that $F_{k*}k$ admits no nonzero continuous functionals $F_{k*}k \rightarrow k$, 
we construct
a restricted (hence convergent) power series whose image under the Frobenius splitting
is a power series whose coefficients have norms growing factorially. Therefore, this image
cannot be a convergent power series using Remark \ref{rem:exp}.
\end{proof}

\begin{remark}
Let $\fm$ be the maximal ideal of $T_n(k)$ generated by the indeterminates.
Then, one can show the following are equivalent:
\begin{enumerate}[label=$(\roman*)$,ref=\roman*]
  \item\label{rem:tatelocfsplit}
    $(T_n(k))_\fm$ is Frobenius split (resp.\ has a nonzero $p^{-1}$-linear map)
    for each integer $n > 0$.
  \item\label{rem:tatehenselfsplit}
    $(T_n(k))_\fm^h$ is Frobenius split (resp.\ has a nonzero $p^{-1}$-linear
    map) for each integer $n > 0$.
  \item\label{rem:fsplitklinmap}
    There exists a nonzero continuous $k$-linear map $f\colon
    F_{k*}k \to k$.
\end{enumerate}
$(\ref{rem:fsplitklinmap}) \Rightarrow (\ref{rem:tatelocfsplit})$ follows from
Theorem \ref{thm:refined-mainthm} and localization, and
$(\ref{rem:tatelocfsplit}) \Rightarrow (\ref{rem:tatehenselfsplit})$ follows by
base extension since the relative Frobenius of the map from a local
ring to its Henselization is an isomorphism. The latter assertion follows by
\cite[\href{https://stacks.math.columbia.edu/tag/097N}{Tag 097N} and
\href{https://stacks.math.columbia.edu/tag/0F6W}{Tag 0F6W}]{stacks-project}
because the Henselization of a local ring $(R, \fm)$ is a filtered colimit of 
\'etale $R$-algebras.
The implication $(\ref{rem:tatehenselfsplit}) \Rightarrow
(\ref{rem:fsplitklinmap})$ follows from the proof of
$(\ref{thm:tatefsplitn1})\Rightarrow(\ref{thm:tatefsplitklinmap})$ in Theorem
\ref{thm:refined-mainthm}.
This is because we have $k$-algebra inclusions
\[
  T_n \subseteq (T_n)_\fm \subseteq
  (T_n)_\fm^h \subseteq K_n
\]
by the
universal property of localization applied to $T_n \hookrightarrow K_n$ and
the Henselian property of $K_n$ (see Theorem
\ref{prop:properties-K_n}$(\ref{conv-Hens})$), and also because of the
fact that for a Frobenius splitting 
$\Phi: F_{(T_n)_\fm^h*} (T_n)_\fm^h \rightarrow (T_n)_\fm^h$, if 
the composition
\[
F_{k*}k \hooklongrightarrow F_{(T_n)_\fm^h*} (T_n)_\fm^h  \overset{\Phi}{\longrightarrow} (T_n)_\fm^h \longtwoheadrightarrow \frac{(T_n)_\fm^h}{\fm (T_n)_\fm^h} \overset{\sim}{\longrightarrow} k
\]
is not continuous, then
we can construct a restricted power series (i.e. an element of $T_n$) whose image under $\Phi$ 
does not even land
in $K_n$.
\end{remark}

\par We also obtain an analogue of Corollary \ref{cor:F-split-usually} for $K_n$.

\begin{corollary}
\label{cor:Kn-F-split-usually}
  Let $(k,\abs{})$ be a non-Archimedean field of
  characteristic $p > 0$.
For every $n > 0$, the convergent power series ring $K_n(k)$ is Frobenius split in the
following cases:
\begin{enumerate}[label=$(\roman*)$,ref=\roman*]
	\item $(k,\abs{})$ is spherically complete.
	\item $k^{1/p}$ has a dense $k$-subspace $V$ 
	which has a countable $k$-basis, hence in particular
	if $[k^{1/p}:k] < \infty$.
	\item $\abs{k^\times}$ is not discrete,
	and the norm on $k^{1/p}$ is polar.
\end{enumerate}
\end{corollary}
\begin{proof}
Lifting $\id_k\colon k \rightarrow k$ to a 
continuous $k$-linear map $k^{1/p} \rightarrow k$ using Theorem \ref{thm:Hahn-Banach}, 
the Corollary then follows
by Theorem \ref{prop:local-analogue}.
\end{proof}

\section{Gabber's example of a non-Archimedean field\texorpdfstring{\\}{ }with
no nonzero 
continuous \texorpdfstring{$p^{-1}$}{p\textasciicircum-1}-linear maps}
\label{sec:troublesome-field}
Following ideas of Ofer Gabber, we now construct 
a non-Archimedean field $(k,\abs{})$ of prime characteristic $p
> 0$ such that $F_{k*}k$ has no nonzero continuous $k$-linear functionals.
The possibility of the existence of such examples
is suggested by Gerritzen \cite{Ger67} and Kiehl \cite{Kie69} (see also \cite[p.\ 63]{BGR84}),
and while similar constructions of valuative fields with infinite
$p$-degree had been studied by Blaszczok and Kuhlmann \cite{BK15},
the connection with the existence of continuous functionals was not made
(see Remark \ref{rem:BKfields}).

\par Recall from the Notation subsection in Section \ref{sect:intro} that we
can identify the Frobenius map $k \to F_{k*}k$ for $k$ as above
with the inclusion $k
\hookrightarrow k^{1/p}$ of $k$ into the field $k^{1/p}$ of $p$-th roots
of elements in $k$.
The field $k^{1/p}$
has a unique norm extending the one on $k$ (see Remark
\ref{rem:extending-norms}$(\ref{alg-ext-norm})$), and 
we will continue denoting this norm on $k^{1/p}$ by $\abs{}$.
In this new notation, we want
to show that there exists a non-Archimedean field $k$ with no nonzero continuous 
functionals $k^{1/p} \rightarrow k$.

The idea behind the example is to create a non-Archimedean field $(k,\abs{})$
 which is not spherically complete, but such that $k$ admits
an extension $k \hookrightarrow \ell \hookrightarrow k^{1/p}$ where $\ell$
is spherically complete under the restriction of the norm on $k^{1/p}$ to $\ell$. One 
then gets the desired result by applying Proposition \ref{prop:trivial-dual-space}.

\par We first review some standard constructions of fields associated to an \emph{additive} value
group $\Gamma$.

\begin{definition}[see {\citeleft\citen{Poo93}\citemid \S3\citepunct
  \citen{Efr06}\citemid \S\S2.8--2.9\citeright}]
  Let $K$ be a field, and let $\Gamma \subseteq \RR$ be an additive subgroup.
  The \textsl{field of generalized rational functions} is the
  field of fractions $K(t^\Gamma)$ of the ring of generalized polynomials
  \[
    K[t^\Gamma] \coloneqq \Set[\Bigg]{\sum_{\gamma \in \Gamma} a_\gamma t^\gamma
    \given \text{the set}\ \Set{\gamma \in \Gamma \given a_\gamma \ne 0}\
    \text{is finite}}.
  \]
  The field $K(t^\Gamma)$ embeds inside the \textsl{Hahn series field}
  \[
    K((t^\Gamma)) \coloneqq \Set[\Bigg]{\sum_{\gamma \in \Gamma} a_\gamma
    t^\gamma \given \text{the set}\ \Set{\gamma \in \Gamma \given a_\gamma \ne
    0}\ \text{is well-ordered}}.
  \]
  These fields are compatibly valued with the analogue of the Gauss norm\footnote{Note 
  that even though $\Gamma$ is written additively, the norm
  on $K((t^\Gamma))$ is a multiplicative valuation.}:
  \[
    \norm[\Bigg]{\sum_{\gamma \in \Gamma} a_\gamma t^{\gamma}} \coloneqq
    \max_{a_\gamma \ne 0}\bigl\{e^{-\gamma}\bigr\}.
  \]
  The value groups of $K(t^\Gamma)$ and $K((t^\Gamma))$ are clearly
  $e^{-\Gamma}$,
  and their residue fields are both $K$.
  
  The Hahn series field $K((t^\Gamma))$ is spherically complete (hence also complete)
  with respect to the analogue of the Gauss norm \cite[Thm.\ 1]{Poo93}.
  It is also called a Mal\cprime cev--Neumann field (see \cite[p.\ 88]{Poo93}).
\end{definition}

\begin{theorem}
\label{thm:main-example}
  Let $K$ be a field of characteristic $p > 0$, and let $\Gamma \subseteq \RR$
  be an additive subgroup such that $\Gamma/p\Gamma$ is infinite.
  Consider the compositum
  \[
    M \coloneqq K((t^{p\Gamma})) \mathbin{.} K(t^\Gamma) \subseteq
    K((t^\Gamma))
  \]
  of fields over $K(t^{p\Gamma})$ with norm induced by that on $K((t^\Gamma))$.
  Then, the following properties hold:
  \begin{enumerate}[label=$(\roman*)$,ref=\roman*]
    \item\label{thm:main-example-boundedcoset} There is a bounded sequence 
    $(r_i)_{i \in \ZZ_{\geq 0}}$
    of pairwise distinct coset representatives of $p\Gamma$ in $\Gamma$.
    \item\label{thm:main-example-algext} $M$ is the algebraic extension of
      $K((t^{p\Gamma}))$ consisting of Hahn series in $K((t^\Gamma))$
      whose exponents lie in finitely many cosets of $p\Gamma$ in $\Gamma$.
    \item\label{thm:main-example-notsphercomp} The completion $\widehat{M}$ of
      $M$
      is not spherically complete.
    \item\label{thm:main-example-nop-1maps} There are no nonzero continuous
      $\widehat{M}$-linear functionals $\widehat{M}^{1/p} \to \widehat{M}$.
  \end{enumerate}
\end{theorem}
\noindent We first explicitly describe an additive subgroup $\Gamma$ satisfying the
hypotheses above.

\begin{example}
  A simple example of a subgroup $\Gamma \subseteq R$ satisfying the
  hypothesis of Theorem \ref{thm:main-example} is
  \[
    \Gamma = \sum_{i \in \ZZ_{\geq 0}} \ZZ \cdot \{r_i\} \subseteq \RR
  \]
  generated by a decreasing sequence of positive real
  numbers $r_i$ such that $r_i \to 0$ as $i \to \infty$
  and such that $\{r_i\}_{i \in
  \ZZ_{\ge0}}$ is linearly independent over $\QQ$.
  Sequences of this form appear in \cite[\S5]{Kap42}, where Kaplansky constructs
  examples of non-Archimedean fields with non-unique spherical completions.
\end{example}

\par We now prove Theorem \ref{thm:main-example}.

\begin{proof}[Proof of Theorem \ref{thm:main-example}]
We first show $(\ref{thm:main-example-boundedcoset})$. Since $\Gamma/p\Gamma$ is infinite,
the group $\Gamma$ is not discrete (that is, isomorphic to $\ZZ$), and hence
$p\Gamma$ is also not discrete.
Consequently, both groups are dense in $\RR$. Choose any sequence 
$(r_i)_{i \in \ZZ_{> 0}}$ of representatives of pairwise distinct cosets of 
$p\Gamma$ in $\Gamma$.
By the density of $p\Gamma$ in $\RR$, for each $r_i$, there exists $f_i \in p\Gamma$
such that $\abs{r_i - f_i} < 1$. Then, replacing $(r_i)_i$ by $(r_i - f_i)_i$
gives a sequence of  representatives of pairwise distinct cosets 
$p\Gamma$ in $\Gamma$ such that
$\{r_i - f_i\}_i \subseteq (-1,1)$, proving $(\ref{thm:main-example-boundedcoset})$.

  \par We next show $(\ref{thm:main-example-algext})$.
  Note that $M$ can be
identified as a subfield of $K((t^\Gamma))$ by adjoining to $K((t^{p\Gamma}))$ 
the $p$-th roots $t^\gamma$ of elements of the form $t^{p\gamma} \in K((t^{p\Gamma}))$, 
for $\gamma \in \Gamma$. 
Moreover, since every element of $M$
lies in a subfield $K((t^{p\Gamma}))(t^{\gamma_1},t^{\gamma_2},\ldots,t^{\gamma_n})$,
for some finite set of elements $\gamma_1,\gamma_2,\ldots,\gamma_n \in \Gamma$,
it follows that $M$ consists of those Hahn series of $K((t^\Gamma))$
whose exponents lie in finitely many cosets of $p\Gamma$ in $\Gamma$.

\par We now show $(\ref{thm:main-example-notsphercomp})$.
We first claim that the extensions
\begin{equation}\label{eq:immext}
  M \subseteq \widehat{M} \subseteq K((t^\Gamma))
\end{equation}
are immediate extensions of real-valued fields.
For this, it suffices to show that $M \subseteq K((t^\Gamma))$ is immediate.
The field $M$ has the same value group 
as $K((t^\Gamma))$ (the norm of the $p$-th root of $t^{p\gamma}$
equals $e^{-\gamma}$, for every $\gamma \in \Gamma$), and both $M$ and
$K((t^\Gamma))$ have residue field $K$, showing that $M \subseteq K((t^\Gamma))$
is immediate.
Since spherically complete fields do not admit proper immediate extensions by
Remark \ref{rem:spherical-comp-facts}$(\ref{maximal-extensions})$, 
to show that $\widehat{M}$ is not spherically complete, it therefore suffices
to show that $\widehat{M}$ is a proper subfield of $K((t^\Gamma))$ in the
sequence \eqref{eq:immext} of immediate extensions.
By $(\ref{thm:main-example-boundedcoset})$, there is a bounded sequence $(r_i)_i$ 
of representatives of pairwise distinct cosets of
$p\Gamma$ in $\Gamma$.
By passing to a subsequence we may assume $(r_i)_i$ is strictly increasing or
strictly decreasing.
After possibly replacing every
$r_i$ by $-r_i$, we may further assume that the sequence $(r_i)$ is strictly
decreasing and bounded below by a real number $r \in \RR$.
The set $\{-r_i\}_{i \in \ZZ_{\ge0}}$ is then well-ordered and bounded above by
$-r$.
We claim the Hahn series
\[
  f = \sum_{i \in \ZZ_{\ge0}} t^{-r_i} \in K((t^\Gamma))
\]
does not lie in $\widehat{M}$.
For this, it suffices to show that elements in $M$ are bounded away from $\sum_i
t^{-r_i}$.
Let $g \in M$ be arbitrary.
Since the $r_i$ do not lie in $p\Gamma$, the series $g \in M$ can only
contain finitely many of the $-r_i$ as exponents by
$(\ref{thm:main-example-algext})$.
Let $i_g$ be such that $-r_{i_g}$ does not appear as an exponent
in $g$.
We therefore see that
\[
  \norm{f - g} = \norm[\Bigg]{t^{-r_{i_g}} + \sum_{\gamma \ne
  -r_{i_g}} b_\gamma t^\gamma} \ge e^{r_{i_g}} > e^{r}
\]
for some $b_\gamma \in K$, where the middle inequality follows by definition of
the Gauss norm $\norm{}$ on $K((t^\Gamma))$, and the last inequality follows
from the fact that the sequence $(r_i)$ is strictly decreasing and 
bounded below
by $r$.
Therefore, the ball of radius $e^{r}$ centered at $f$ contains no element of $M$.
Hence, $M$ is not dense in $K((t^\Gamma))$, and consequently   
$\widehat{M}$ is a proper subfield of $K((t^\Gamma))$ because $M$ is 
dense in $\widehat{M}$.
\par Finally, we show $(\ref{thm:main-example-nop-1maps})$.
  Since $\widehat{M}$ is not spherically complete by $(\ref{thm:main-example-notsphercomp})$,
it follows that there are no nonzero continuous $\widehat{M}$-linear maps
$K((t^\Gamma)) \rightarrow \widehat{M}$ by Proposition \ref{prop:trivial-dual-space}.
Since we have the inclusions
\[
  \widehat{M} \subseteq K((t^\Gamma)) \subseteq \widehat{M}^{1/p},
\]
if $f\colon \widehat{M}^{1/p} \rightarrow \widehat{M}$ 
is a continuous $\widehat{M}$-linear
map such that $f(x) \neq 0$, then for every nonzero $a \in K((t^\Gamma))$,
the composition
\[
K((t^\Gamma)) \hooklongrightarrow \widehat{M}^{1/p} \xrightarrow{-\cdot a^{-1}x} \widehat{M}^{1/p}
\overset{f}{\longrightarrow} \widehat{M}
\]
is a continuous $\widehat{M}$-linear map such that $a \mapsto f(x) \neq 0$. But
this contradicts the fact that $K((t^\Gamma))$ has no nonzero continuous 
$\widehat{M}$-linear
functionals.
\end{proof}

\begin{remark}\label{rem:BKfields}
  In \cite[Thm.\ 1.6]{BK15}, Blaszczok and Kuhlmann give a more general
  construction of fields similar to those in Theorem \ref{thm:main-example},
  although in our special setting the arguments are simpler.
  For suitable spherically complete non-Archimedean fields $(K,\abs{})$ of
  characteristic $p > 0$, Blaszczok and Kuhlmann construct a non-Archimedean
  field $L$ fitting into a sequence
  \[
    K \subseteq L \subseteq K^{1/p}
  \]
  of immediate extensions, such that the completion $\widehat{L}$ (denoted by
  $L^c$ in \cite[pp.\ 211--213]{BK15}) of $L$ is strictly
  contained in $K^{1/p}$.
  The argument in $(\ref{thm:main-example-notsphercomp})$ shows that
  $\widehat{L}$ is not spherically complete, and the argument in
  $(\ref{thm:main-example-nop-1maps})$ shows that there are no nonzero
  continuous $\widehat{L}$-linear functionals $\widehat{L}^{1/p} \to
  \widehat{L}$.
  We can therefore replace the field $\widehat{M}$ constructed in Theorem
  \ref{thm:main-example} with Blaszczok and Kuhlmann's field $\widehat{L}$ in
  the proofs of Theorem \ref{thm:mainthm} and Corollary
  \ref{cor:no-p-1-linear-maps} below.
\end{remark}

\par The proofs of Theorem \ref{thm:mainthm} and Corollary \ref{cor:no-p-1-linear-maps}
are now a simple matter of interpreting Theorem \ref{thm:refined-mainthm}
in light of the construction in Theorem \ref{thm:main-example}.

\begin{customthm}{A}
\label{thm:mainthm}
For every prime $p > 0$,
there exists a complete non-Archimedean field $(k,\abs{})$ of
characteristic $p$ such that the Tate algebra 
$T_n(k) \coloneqq k\{ X_1,X_2,\ldots,X_n \}$
is not Frobenius split for each $n > 0$.
In fact, $T_n(k)$ admits no nonzero $T_n(k)$-linear maps
$F_{T_n(k)*}T_n(k) \to T_n(k)$ for each $n > 0$.
\end{customthm} 

\begin{proof}
Take $k$ to be the non-Archimedean field $\widehat{M}$
constructed in Theorem \ref{thm:main-example}.
We can then use Theorem \ref{thm:refined-mainthm}
to conclude that the Tate algebras $T_n(k)$ cannot admit any nonzero $p^{-1}$-linear
maps for each $n > 0$.
\end{proof}

\begin{remark}
One obtains an analogue of Theorem \ref{thm:mainthm} for the
regular local convergent power series rings $K_n(k)$
by using Theorem \ref{prop:local-analogue}.
\end{remark}

\begin{customcor}{C}
\label{cor:no-p-1-linear-maps}
There exists an excellent Euclidean domain $R$ of characteristic
$p > 0$ such that $R$ admits no nonzero $R$-linear maps $F_{R*}R \rightarrow R$.
Moreover, one can choose $R$ to be local and Henselian as well.
\end{customcor}

\begin{proof}
Take $k$ to be
the field $\widehat{M}$ constructed in Theorem \ref{thm:main-example}
and $R$ to be the Tate algebra $T_1(k)$. Then, $R$ is a Euclidean domain by
Theorem \ref{thm:tateprops}$(\ref{thm:T_1euclid})$, and admits no nonzero $R$-linear maps
$F_{R*}R \rightarrow R$ by Theorem \ref{thm:mainthm}. To get a local and Henselian
Euclidean domain $R$, one can choose $R = K_1(k)$ and then apply 
Theorem \ref{prop:local-analogue}.
\end{proof}

\begin{remark}
\label{rem:as-close-to-complete}
A well-known argument using Matlis duality shows that
a Noetherian complete local ring that is $F$-pure is always Frobenius split \cite[Lem.\
1.2]{Fed83}. 
Thus, Corollary \ref{cor:no-p-1-linear-maps} provides a stark
contrast with the complete case, since it shows that 
Question \ref{ques:fpureisfsplit} fails 
even for excellent local rings that behave the most like
complete local rings, namely those that are Henselian.
\end{remark}

\begin{remark}
Let $k$ be the non-Archimedean field $\widehat{M}$ from Theorem \ref{thm:main-example}.
Since there are no nonzero continuous $k$-linear maps 
$F_{k*}k \rightarrow k$, the valuation ring $k^\circ$ has no nonzero
$k^\circ$-linear maps $F_{k^\circ*}k^\circ \rightarrow k^\circ$. Indeed, any nonzero
$k^\circ$-linear map $f\colon F_{k^\circ*}k^\circ \rightarrow k^\circ$ extends to nonzero
$k$-linear map $\widetilde{f}\colon F_{k*}k \rightarrow k$ at the level of fraction fields.
Since $\widetilde{f}(B_{1}(0)) \subseteq B_1(0)$ by virtue of $\widetilde{f}$ 
being an extension of $f$,
it follows that $\widetilde{f}$ is continuous by
Lemma \ref{lem:continuous}.
But this is impossible by
Theorem \ref{thm:main-example}.
This answers a question raised by the first author in \cite[p.\ 25]{Dat}
about the existence of non-Frobenius split complete rank $1$
valuation rings $k^\circ$ of a field $k$ for which the extension $k
\hookrightarrow k^{1/p}$ is not immediate (the group $\Gamma$ in this case is
not $p$-divisible).
\end{remark}

\end{document}